\documentclass[11pt, reqno]{amsart}
\usepackage[utf8]{inputenc}
\usepackage[T1]{fontenc}
\usepackage{amssymb}
\usepackage{graphics}
\usepackage{xcolor}
\usepackage{enumitem}
\usepackage[pagebackref, colorlinks = true, linkcolor = teal, urlcolor  = teal, citecolor = violet]{hyperref}
\usepackage[margin=1in]{geometry}
\usepackage{cleveref}
\crefname{equation}{Eq.}{Eqs.}
\usepackage{bbm}

\newtheorem{theorem}{Theorem}[section]
\newtheorem{definition}[theorem]{Definition}
\newtheorem{proposition}[theorem]{Proposition}
\newtheorem{corollary}[theorem]{Corollary}
\newtheorem{lemma}[theorem]{Lemma}
\newtheorem{remark}[theorem]{Remark}
\newtheorem{example}[theorem]{Example}

\newcommand{\id}{{\rm Id}}
\newcommand{\rank}{\operatorname{rank}}

\newcommand{\tr}{{\operatorname{tr}}}
\newcommand{\Lip}{{\rm Lip}}
\renewcommand{\d}{{\rm d}}

\newcommand{\pp}{{\mathbb P}}
\newcommand{\ee}{{\mathbb E}}
\newcommand{\ppch}{\pp^{\mathrm{ch}}}
\newcommand{\eech}{\ee^{\mathrm{ch}}}

\newcommand{\rr}{{\mathbb R}}
\newcommand{\nn}{{\mathbb N}}
\newcommand{\cc}{{\mathbb C}}
\renewcommand{\P}{{\mathrm P}}
\newcommand{\proj}{\pi}

\newcommand{\supp}{\operatorname{supp}}

\newcommand{\irr}{{\hyperlink{irr}{\bf($\phi$-Erg)}}}
\newcommand{\pur}{{\hyperlink{pur}{\bf(Pur)}}}

\newcommand{\inv}{_{\mathrm{inv}}}
\newcommand{\dist}{d}
\newcommand{\lps}{\langle}
\newcommand{\rps}{\rangle}
\newcommand{\as}{\operatorname{-}\mathrm{a.s.}}
\renewcommand{\ae}{\operatorname{-}\mathrm{a.e.}}

\newcommand\pt[1]{^{\otimes #1}}

\newcommand{\sta}{\P(\cc^k)}
\newcommand{\staalg}{\mathcal{B}}
\newcommand{\out}{\Omega}
\newcommand{\outalg}{\mathcal{O}}
\newcommand{\joint}{\sta\times \out}
\newcommand{\jointalg}{\mathcal{J}}

\begin{document}

\title{Invariant Measure for Quantum Trajectories}

\author{T. Benoist}
\address{Institut de Math\'ematiques de Toulouse, \'Equipe de Statistique et Probabilit\'es,
Universit\'e Paul Sabatier, 31062 Toulouse Cedex 9, France}
\email{tristan.benoist@math.univ-toulouse.fr}
\author{M. Fraas}
\address{Instituut voor Theoretische Fysica, KU Leuven, Belgium}
\email{martin.fraas@gmail.com}
\author{Y. Pautrat}
\address{Laboratoire de Math\'ematiques d'Orsay, Univ. Paris-Sud, CNRS, Universit\'e Paris-Saclay, 91405 Orsay, France}
\email{yan.pautrat@math.u-psud.fr}
\author{C. Pellegrini}
\address{Institut de Math\'ematiques de Toulouse, \'Equipe de Statistique et Probabilit\'es,
Universit\'e Paul Sabatier, 31062 Toulouse Cedex 9, France}
\email{clement.pellegrini@math.univ-toulouse.fr}

\subjclass[2000]{}

\begin{abstract}
We study a class of Markov chains that model the evolution of a
quantum system subject to repeated measurements. Each Markov chain in this
class is defined by a measure on the space of matrices. It
is then given by a random product of correlated matrices taken from
the support of the defining measure. We give natural conditions on
this support that imply that the Markov chain admits a unique invariant
probability measure. We moreover prove the geometric convergence
towards this invariant measure in the Wasserstein metric.  Standard
techniques from the theory of products of random matrices cannot be
applied under our assumptions, and new techniques are developed, such as maximum likelihood-type estimations.
\end{abstract}

\date{\today}

\maketitle

\tableofcontents
\section{Introduction} \label{sec_intro}
We consider a complex vector space $\mathbb{C}^k$ and its projective space $\P(\mathbb{C}^k)$ equipped with its Borel $\sigma$-algebra $\staalg$. 
For a non zero vector $x\in\mathbb C^k$, we denote $\hat x$ the corresponding equivalence class of $x$ in $\P(\mathbb C^k)$.
For a linear map  $v\in M_k(\mathbb C)$ we denote $v\cdot \hat{x} $ the element of the projective space represented by $v\,x$ whenever $v\,x\neq 0$. We equip $M_k(\mathbb C)$ with its Borel $\sigma$-algebra and let $\mu$ be a measure on $M_k(\mathbb C)$ with a finite second moment, $\int_{M_k(\mathbb C)} \|v\|^2\,\d\mu(v)<\infty$, that satisfies the stochasticity condition
\begin{equation}
\label{eq:stochastic family}
\int_{M_k(\mathbb C)} v^* v \,\mathrm{d} \mu(v) = \id_{\cc^k}. 
\end{equation}

In this article we are interested in particular Markov chains $(\hat{x}_n)$  on $\P(\mathbb{C}^k)$, defined by 
\[
\hat x_{n+1}=V_n\cdot \hat x_{n},
\]
where $V_n$ is a $M_k(\mathbb C)$-valued random variable with a probability density $ ||v x_n||^2/||x_n||^2 \mathrm{d} \mu(v).$ More precisely, such a Markov chain is associated with the transition kernel given for a set $S\in\staalg$ and $\hat x\in \P(\mathbb C^k)$ by
\begin{equation} \label{eq_deftranskernel}
\Pi(\hat x,S)=\int_{M_k(\cc)} \mathbf{1}_{S}(v\cdot \hat x)  \|v x\|^2 \mathrm{d} \mu(v),
\end{equation}
where $x$ is an arbitrary normalized vector representative of $\hat{x}$. Note that the normalization condition \eqref{eq:stochastic family} imposed on $\mu$ is equivalent to the conservation of probability,  $\Pi\big(\hat x, \P(\mathbb{C}^k)\big) =1$. We are interested in the large-time distribution of $(\hat x_n)$.

Note that $\hat x_n$ can be written as $$\hat x_n = V_{n}\ldots V_{1}\cdot \hat x_0$$ so that the study of $\hat x_n$ can be formulated in terms of random products of matrices. Markov chains associated to random products of matrices were studied in a more general setting where the weight appearing in the transition kernel \eqref{eq_deftranskernel} is proportional to $\|v x\|^s$ for some $s \geq 0$, instead of $\|v x\|^2$. The classical case of products of independent, identically distributed random matrices pioneered by Kesten, Furstenberg and Guivarc'h corresponds to $s = 0$. In that case, for i.i.d. invertible random matrices $Y_1,Y_2,\ldots$, denoting $S_n=Y_n\ldots Y_1$, one is usually interested in the asymptotic properties of
$$\log\| S_n x\|,$$ for any $x\neq 0$. In particular, a law of large numbers, a central limit theorem and a large deviation principle have been obtained for this quantity, under contractivity and strong irreducibility assumptions \cite{Furstenberg, GuRa85, LePage}. Such results are closely linked to the uniqueness of the invariant measure of the Markov chain 
$$\hat x_n=S_n\cdot \hat x.$$  
These results were generalized to the case $s >0$ in \cite{Guivarch2016}. Our framework corresponds to the case $s=2$; in this case, and with the additional assumption \eqref{eq:stochastic family}, we provide a new method to study this Markov chain, and use it to derive the above results without assuming invertibility of the matrices, and with an optimal irreducibility assumption. We compare our approach with respect to that of \cite{Guivarch2016} at the end of this section.

The method that we employ is motivated by an interpretation of this process as statistics of a quantum system being repeatedly indirectly measured. Let us expand on this as we introduce more notation and terminology. The set of states of a quantum system described by a finite dimensional Hilbert space $\cc^k$ is the set of density matrices $\mathcal D_k:=\{\rho\in M_k(\cc)\ |\ \rho\geq 0,\ \tr\,\rho=1\}$. This set is convex and the set of its extreme points is called the set of pure states. This latter  set is in one to one correspondence with the projective space $\sta$ by the bijection $\sta\ni\hat x\mapsto \proj_{\hat x}\in\mathcal D_k$ with $\proj_{\hat x}$ the orthogonal projector on the corresponding ray in $\cc^k$.

The time evolution of the system conditioned on a measurement outcome is encoded in a matrix $v$ that updates the state of the system. The support of $\mu$ is endowed with the meaning of the possible updates, and the system is updated according to $v$ with a probability density $\tr(v \rho v^*)\, \d \mu(v)$. Given $v$, a state $\rho$ is mapped to a state $v \rho v^*/\tr(v \rho v^*)$. Iterating this procedure defines a random sequence $(\rho_n)$ in $\mathcal D_k$ called a quantum trajectory: after $n$ measurements with resulting matrices $v_1,\dots,v_n$ the state of the system becomes 
\begin{equation}\label{qu:tr} \rho_n=\frac{v_{n}\ldots v_{1}\rho_0 v_{1}^* \ldots v_{n}^*}{\tr(v_{n}\ldots v_{1}\rho_0 v_{1}^* \ldots v_{n}^*)} 
\end{equation}
where $(v_1,\ldots,v_n)$ has probability density $\tr(v_{n} \dots v_{1} \rho_0 v_{1}^*\ldots v_{n}^*)\,\d\mu^{\otimes n}(v_1,\ldots,v_n)$. In other words, the process \Cref{qu:tr}  describes an evolution of a repeatedly measured  quantum system.

A key result in the theory of quantum trajectories is a purification theorem obtained by K\"ummerer and Maassen \cite{Maassen} showing  that quantum trajectories $(\rho_n)$ defined on $\mathcal{D}_k$ almost surely approach the set of pure states, the extreme points of $\mathcal{D}_k$ if and only if the following purification condition is satisfied:
\hypertarget{pur}{}
\begin{description}
\item[\pur] Any orthogonal projector $\proj$ such that for any $n\in\nn$, $\proj v_1^*\ldots v_n^* v_n\ldots v_1 \proj \propto \proj$  for $\mu^{\otimes n}$-almost all $(v_1,\ldots,v_n)$, is of rank one
\end{description}
(we write $X \propto Y$ for $X,Y$ two operators if there exists $\lambda\in \cc$ such that $X=\lambda Y$).

Under this assumption, the long-time behavior of the Markov chain is essentially dictated by its form on the set of pure states, i.e.\ for $\rho_0=\proj_{\hat x_0}$. It is an immediate observation that
\begin{equation} \label{eq_xnrhoncorresp}
\tr(v\proj_{\hat x_0}v^*) = \|v x_0\|^2, \quad
\frac{v\proj_{\hat x_0} v^*}{\tr(v\proj_{\hat x_0}v^*)} = \proj_{v\cdot \hat x_0},
\end{equation}
for all $v\in M_k(\mathbb C)$. This way our Markov chain $(\hat x_n)$ corresponds to the quantum trajectory $(\rho_n)$ described above when $\rho_0$ is a pure state $\proj_{\hat x_0}$.

Although ideas underlying our method are based on the connection of $(\hat x_n)$, with this physical problem, we will not explicitly use it in the paper. The notion of quantum trajectory originates in quantum optics \cite{carmichael}, and Haroche's Nobel prize winning experiment \cite{Haroche} is arguably the most prominent  example of a system described by the above formalism. The reader interested in the involved mathematical structures might consult for example the review book \cite{Holevo} or the pioneering article \cite{Maassen}.
\smallskip

We will show that under condition \pur, the set of all invariant measures of the Markov chain \eqref{qu:tr} can be completely classified, depending on the operator $\phi$ on $\mathcal{D}_k$ describing the average evolution:
\begin{equation}\label{channel}
\phi(\rho)=\int_{M_k(\mathbb C)} v \rho v^* \, \d\mu(v).
\end{equation}
The map $\phi$ on $\mathcal D_k$ is completely positive and trace-preserving.\footnote{Complete positivity is stronger than positivity; namely $\phi$ is completely positive if $\phi\otimes \id_{M_n(\cc)}$ is positive for all $n\in\nn$.} Such a map is often called a quantum channel (see e.g.\  \cite{wolftour}). %HeZi12,Pe08
It has in particular the property of mapping states to states. Brouwer's fixed point Theorem shows that there exists an invariant state, i.e.\ $\rho\in \mathcal D_k$ such that $\phi(\rho)=\rho$. A necessary and sufficient algebraic condition for uniqueness of this invariant state is (see e.g.\ \cite{EHK,wolftour,CarbonePautrat})
\hypertarget{irr}{}
\begin{description}
\item[\irr] There exists a unique minimal non trivial $\operatorname{supp} \mu$-invariant subspace $E$ of $\cc^k$.
\end{description}
 If \irr\ holds with $E=\cc^k$, then $\phi$ is said irreducible. We chose the name \irr\ to avoid confusion with the notion of irreducibility for Markov chains. 
We moreover emphasize that we call this assumption \irr\ because it relies only on $\phi$ and not on the different operators $v$ in the support of $\mu$: an equivalent statement of \irr\ is that there exists a unique minimal nonzero orthogonal projector $\proj$ such that $\phi(\proj)\leq \lambda \proj$ for some $\lambda \geq 0$ (see e.g.\ \cite{schrader}).
\medskip

We now state the main result of the paper:
\begin{theorem}\label{thm:uniqueness}
Assume that $\mu$ satisfies assumptions \pur\ and \irr. Then, the transition kernel $\Pi$ has a unique invariant probability measure $\nu\inv$ and there exists $m\in\{1,\ldots,k\}$, $C>0$ and $0<\lambda<1$ such that for any probability measure $\nu$ over $\big(\sta,\mathcal B\big)$,
\begin{equation}\label{eq:conv with period}
W_1\left(\frac1m\sum_{r=0}^{m-1} \nu\Pi^{mn+r}, \nu\inv\right)\leq C \lambda^n,
\end{equation}
where $W_1$ is the Wasserstein metric of order $1$.
\end{theorem}

The Wasserstein metric is constructed with respect to a natural metric on the complex projective space. This metric is defined, for $\hat{x},\,\hat{y}$ in $\sta$, by
\begin{equation}
\label{eq:metric}
\dist(\hat{x},\hat{y}) = \big(1-|\lps x,y\rps|^2\big)^{\frac12}, 
\end{equation}
where $x,\,y$ are unit length representative vectors and $\langle\,\cdot\,,\,\cdot\,\rangle$ is the canonical hermitian inner product on $\mathbb{C}^k$.

\medskip

Let us now compare our results to those of the article \cite{Guivarch2016} of Guivarc'h and Le Page. They consider a probability distribution $\mu$ with support in $GL_k(\mathbb C)$, without requiring the normalization condition \eqref{eq:stochastic family},  and study the transition kernel on $\sta$ given, for $S\in \mathcal B$, by
\[
\Pi_{s}(\hat x,S) \propto \int_{M_k(\mathbb C)}\mathbf{1}_S(v\cdot \hat x)  \|v x\|^s \mathrm{d} \mu(v).
\]
In the case $s=2$, Theorem A of \cite{Guivarch2016} implies the conclusions of \Cref{thm:uniqueness} under two assumptions:
\begin{itemize}
\item strong irreducibility, in the sense that there is no non-trivial finite union of proper subspaces of $\cc^k$ left invariant by all $v\in \mathrm{supp}\,\mu$,
\item contractivity, in the sense that there exists a sequence $(a_n)$ in $T_\mu$, the smallest closed sub-semigroup of $GL_k(\cc)$ containing $\operatorname{supp}\mu$, such that $\lim_{n\to\infty}a_n/\|a_n\|$ exists and is of rank one.
\end{itemize}
It is, however, a simple exercise to prove that strong irreducibility of $\mu$ implies \irr\ with $E=\cc^k$. In addition, if we assume $\operatorname{supp}\mu\subset GL_k(\mathbb C)$ and $\operatorname{supp}\mu$ is strongly irreducible, the equivalence
$$\mbox{\pur}\iff T_\mu\mbox{ is contracting}$$
holds (see Appendix \ref{app:contractivity_iff_pur}).
Our results therefore offer a strong refinement of \cite{Guivarch2016} in the restricted framework of $s=2$ with $\int v^* v \,\d\mu(v)=\id_{\cc^k}$.

\medskip
The article is structured as follows. \Cref{sec:unique} is devoted to the first part of \Cref{thm:uniqueness}, that is the uniqueness of the invariant measure. 
In \Cref{sec:speed} we show the geometric convergence towards the invariant measure with respect to the $1$-Wasserstein metric. 
In \Cref{sec_Lyapunoventropy} we discuss the Lyapunov exponents of the process and relate them to the convergence between the Markov chain and an estimate of the chain used in our proofs. 

\smallskip
\paragraph{\textbf{Notation}} In all of the following, for $x\in\cc^k\setminus\{0\}$, $\hat x$ is its equivalence class in $\P(\cc^k)$ and, for $\hat x$ in $\P(\cc^k)$, $x$ is an arbitrary norm one vector representative of $\hat x$. If e.g.\ $\pp_\nu$ (resp. $\mathbb P^\rho$) is a probability measure (depending on some a priori object $\nu$ (resp. $\rho$)) then $\ee_\nu$\ (resp. $\ee^\rho$) is the expectation with respect to $\pp_\nu$ (resp. $\pp^\rho$). The set $\nn$ represents the set of positive integers $\{1,2,\ldots\}$.%, and $\nn_0$ the set of nonnegative integers $\{0,1,2,\ldots\}$.

 \section{Uniqueness of the invariant measure}
 \label{sec:unique}
This section concerns essentially the first part of \Cref{thm:uniqueness}. More precisely, under {\irr} and \pur\ we show that the Markov chain has a unique invariant measure. We note that an invariant measure always exists since $\P(\mathbb C^k)$ is compact.
\smallskip

We now proceed to introduce some additional notation. We consider the space of infinite sequences $\out:=M_k(\cc)^{\mathbb N}$, write $\omega = (v_1,v_2, \dots)$ for any such infinite sequence, and denote by $\pi_n$ the canonical projection on the  first $n$ components, $\pi_n(\omega)=(v_1,\ldots,v_n)$. Let $\mathcal M$ be the Borel $\sigma$-algebra on $M_k(\cc)$. For $n\in\nn$, let $\outalg_n$ be the  $\sigma$-algebra on $\Omega$ generated by the $n$-cylinder sets, i.e.\ $\outalg_n = \pi_n^{-1}(\mathcal M^{\otimes n})$. We equip the space $\Omega$ with the smallest $\sigma$-algebra $\outalg$ containing $\outalg_n$ for all $n\in \nn$. We let $\mathcal B$ be the Borel $\sigma$-algebra on $\P(\mathbb C^k)$, and denote
\[\jointalg_n=\mathcal B\otimes \outalg_n,\qquad \jointalg=\mathcal B\otimes \outalg.\]
This makes $\big(\P(\mathbb C^k)\times \Omega,\jointalg\big)$ a measurable space. With a small abuse of notation we denote the sub-$\sigma$-algebra $\{\emptyset,\P(\mathbb C^k)\}\times \outalg$ by $\outalg$ and equivalently identify any $\outalg$-measurable function $f$ with a $\jointalg$-measurable function $f$ satisfying $f(\hat{x},\omega) = f(\omega)$.

For $i\in\mathbb N$, we consider the random variables $V_i : \Omega \mapsto  M_k(\mathbb C)$,
\begin{equation}
	V_i(\omega) = v_i \quad \mbox{for} \quad \omega=(v_1,v_2,\ldots), \label{eq:W}
\end{equation}
and we introduce $\mathcal O_n$-mesurable random variables $(W_n)$ defined for all $n\in\mathbb N$ as $$W_n=V_{n}\ldots V_{1}.$$ 

With a small abuse of notation we identify cylinder sets and their bases, and extend this identification to several associated  objects. In particular we identify $O_n\in \mathcal M^{\otimes n}$ with $\pi_n^{-1}(O_n)$, a function $f$ on $\mathcal M^{\otimes n}$  with $f \circ \pi_n$ and a measure $\mu^{\otimes n}$ with the measure $\mu^{\otimes n} \circ \pi_n$.
Since $\mu$ is not necessarily finite, note that we can not extend $(\mu^{\otimes n})$ into a measure on $\Omega$.

Let $\nu$ be a probability measure over $(\sta,\mathcal B)$. We extend it to a probability measure $\mathbb{P}_\nu$ over $(\sta\times\out,\jointalg)$ by letting, for any $S\in \mathcal B$ and any cylinder set $O_n \in \outalg_n$,
\begin{equation} \label{eq_defPnu}
\mathbb{P}_\nu(S \times O_n):=\int_{S\times O_n}  \|W_n(\omega)x\|^2 \d\nu(\hat x) \d\mu\pt n(\omega).
\end{equation}
From relation \eqref{eq:stochastic family}, it is easy to check that the expression \eqref{eq_defPnu} defines a consistent family of probability measures and, by Kolmogorov's Theorem, this defines a unique probability measure $\pp_\nu$ on $\joint$. In addition, the restriction of $\mathbb{P}_\nu$ to $\mathcal B\otimes \{\emptyset,\Omega\}$ is by construction $\nu$.
\smallskip

We now define the random process $(\hat x_n)$. For $(\hat x, \omega)\in \joint$ we define $\hat x_0(\hat x, \omega)=\hat x$. Note that for any $n$, the definition \eqref{eq_defPnu} of $\pp_\nu$ imposes
\[\pp_\nu(W_n x_0 = 0)=0.\]
This allows us to define a sequence $(\hat x_n)_{n\in\mathbb N}$ of $(\jointalg_n)$-adapted random variables on the probability space $(\joint,\jointalg, \pp_\nu)$ by letting
\begin{equation} \label{eq_defxn}
\hat x_n:= W_n\cdot \hat x 
\end{equation}
whenever the expression makes sense, i.e.\ for any $\omega$ such that $W_n(\omega) x\neq 0$, and extending it arbitrarily to the whole of $\Omega$. The process $(\hat x_n)$ on $(\Omega\times \P(\mathbb C^k),\jointalg, \mathbb{P}_\nu)$ has the same distribution as the Markov chain defined by $\Pi$ and initial probability measure $\nu$.

Let us highlight the relation between $\pp_\nu$ and density matrices. To that end, let
\begin{equation} \label{eq_defrhonu}
\rho_\nu:= \mathbb E_\nu(\proj_{\hat x}).
\end{equation}
By linearity and positivity of the expectation, $\rho_\nu\in\mathcal D_k$. Note that, conversely, for a given $\rho\in\mathcal D_k$ there exists $\nu$ (in general non-unique) such that $\rho_\nu= \rho$. For example, if a spectral decomposition of $\rho$ is $\rho=\sum_j p_j \proj_{x_j}$ then necessarily $\sum_j p_j=1$, so that $\nu = \sum_j p_j \delta_{\hat x_j}$ is a probability measure on $\sta$, and it satisfies the desired relation \eqref{eq_defrhonu}.

This relation motivates the following definition of probability measures over $(\out,\outalg)$. For $\rho\in\mathcal{D}_k$  and any cylinder set $O_n \in \outalg_n$, let
\begin{equation} \label{eq_defPrho}
\pp^{\rho}(O_n):=  \int_{O_n}\tr\big(W_n(\omega) \rho \hspace{0.08em}W_n^*(\omega)\big) \mathrm{d} \mu\pt n(\omega). 
\end{equation}
In particular, for any $S\in\staalg$ and $A\in \outalg$,
\begin{equation}\label{eq:desintegration}
\pp_\nu(S\times A)=\int_{S}\pp^{\pi_{\hat x}}(A)\, \d\nu(\hat x).
\end{equation}

The following proposition elucidates further the connections between $\pp_\nu$ and $\pp^{\rho_\nu}$.
\begin{proposition}\label{prop:marginal}
The marginal of $\mathbb{P}_\nu$ on $\outalg$ is the probability measure $\mathbb P^{\rho_\nu}$. 

Moreover, if {\irr} holds, $\pp^{\rho_{\nu_a}}=\pp^{\rho_{\nu_b}}$ for any two $\Pi$-invariant probability measures $\nu_a$ and $\nu_b$.
\label{prop:2.1}
\end{proposition}
\begin{proof}
By construction it is sufficient to check the equality of the measures on cylinder sets. Let $O_n \in \outalg_n$; from the definition of $\mathbb{P}_\nu$,  and trace and integral linearity we have 
\begin{align*}
\mathbb{P}_\nu(\sta\times O_n)&=\int_{\sta\times O_n} \tr\big(W^*_n(\omega)W_n(\omega)\proj_{\hat x}\big)\, \d\nu(\hat x) \d \mu\pt n(\omega)\\
&=\int_{O_n} \tr\big(W_n^*(\omega)W_n(\omega)\int_{\sta} \proj_{\hat x} \,\d\nu(\hat x)\big) \,\d \mu\pt n(\omega)\\
&=\int_{O_n} \tr\big(W_n^*(\omega)W_n(\omega)\rho_\nu\big) \,\d \mu\pt n(\omega).
\end{align*}
The equality between the marginal of $\pp_\nu$ on $\outalg$ and $\pp^{\rho_\nu}$ follows.

If $\nu$ is an invariant measure, on the one hand
$$\mathbb E_\nu(\proj_{\hat x_1})=\mathbb E_\nu(\proj_{\hat x_0})=\rho_\nu.$$
On the other hand,
\begin{align*}
\mathbb E_\nu(\proj_{\hat x_1})=&\int_{\sta\times M_k(\cc)} \frac{v\pi_{\hat x_0} v^*}{\|vx_0\|^2}\|vx_0\|^2\d\nu(\hat x_0)\d\mu(v)\\
=&\int_{M_k(\cc)} v\,\ee_\nu(\proj_{\hat x_0})\,v^*\d\mu(v)\\
=&\phi(\rho_\nu),
\end{align*}
so that  $\rho_\nu$ is a fixed point of $\phi$. 
 
Hence if {\irr} holds, $\rho_\nu$ is the unique fixed point of $\phi$ in $\mathcal D_k$. Hence $\rho_{\nu_a}=\rho_{\nu_b}$ and $\pp^{\rho_{\nu_a}}=\pp^{\rho_{\nu_b}}$ holds.
\end{proof}

In the following we use the measure $\ppch=\pp^{\frac1k\id_{\cc^k}}$ associated to the operator $\id_{\mathbb C^k}/k\in\mathcal D_k$ as a reference measure.
Since for any $\rho\in\mathcal D_k$ there exists a constant $c$ such that $\rho\leq c \frac{\id_{\mathbb C^k}}{k}$ we have
$$\pp^\rho\ll \ppch.$$
The Radon--Nykodim derivative will be made explicit in Proposition \ref{lem:conv proj}.
 To that end, we use a particular $(\outalg_n)$-adapted process. We define a sequence of matrix-valued random variables:
$$M_n:=\frac{W_n^*W_n}{\tr(W_n^*W_n)} \quad \mbox{if} \quad \tr(W_n^*W_n)\neq0$$
and extend the definition arbitrarily whenever $\tr(W_n^*W_n)=0$. The latter alternative appears with probability 0. Indeed $\ppch\big(\tr(W_n^*W_n)=0\big)=0$ and then by the absolute continuity of $\pp^\rho$ with respect to $\ppch$ we have $\pp_\nu\big(\tr(W_n^*W_n)=0\big)=\pp^{\rho_\nu}\big(\tr(W_n^*W_n)=0\big)=0$ for any measure $\nu$. The key property of $M_n$, that we establish in the proof of Proposition \ref{lem:conv proj}, is that it is a $(\outalg_n)$-martingale w.r.t.\ $\ppch$.

%Random variables $M_n$ and $W_n$ are connected by the polar decomposition. As this will be useful in the following, we introduce the decomposition explicitly by
From the existence of a polar decomposition for $W_n$, for each $n$, there exists a unitary matrix-valued random variable $U_n$ such that
\begin{equation}\label{eq:polar}
W_n=U_n\sqrt{\tr(W_n^*W_n)}M_n^{\frac12}.
\end{equation}
This process $(U_n)$ can be chosen to be $(\outalg_n)$-adapted.
%where $(U_n)$ is a sequence of unitary matrix-valued random variables. In fact \eqref{eq:polar} should be understood as a definition of $U_n$ and we note that the process $U_n$ can be chosen in an $\mathcal O_n$-measurable way.
\smallskip

The key technical results about $M_n$ needed for our proofs are summarized in the following proposition.

\begin{proposition}\label{lem:conv proj}\, \hfill
For any probability measure $\nu$ over $(\sta,\mathcal B)$, $(M_n)$
converges $\mathbb{P}_\nu\as$ and in $L^1$-norm to an $\outalg$-measurable random variable $M_\infty$. The  change of measure formula
\begin{eqnarray}
\frac{\d\mathbb P^\rho}{\d \ppch}=k\,\tr(\rho M_\infty)
\end{eqnarray} 
holds true for all $\rho\in\mathcal D_k$.

 Moreover, the measure $\mu$ verifies {\pur} if and only if $M_\infty$ is $\pp_\nu\as$ a rank one projection for any probability measure $\nu$ over $(\sta,\mathcal B)$.
\end{proposition}
\begin{proof}
 We start the proof by showing that $M_n$ is a $\ppch$-martingale. Recall that for all $n\in \mathbb N$ and all $O_n\in \outalg_n$,
$$\ppch(O_n)=\frac1k \int_{O_n} \tr\big(W_{n}^*(\omega)W_n(\omega)\big) \, \d \mu\pt n(\omega).$$ 
From the definition of $W_n$, \Cref{eq:W},
\begin{equation}
\label{eq:MR}
M_{n+1} = \frac{W^*_{n}V^*_{{n+1}}V_{{n+1}} W_n}{\tr\big(W^*_n W_n\big)}\, \frac{\tr\big(W^*_n W_n\big)}{\tr\big(W^*_{n+1} W_{n+1}\big)}.
\end{equation}
This implies that for an arbitrary $\outalg_n$-measurable random variable $Y$
\begin{align*}
\eech(Y M_{n+1}) &= \frac{1}{k} \int_{M_k(\cc)^{n+1}} \frac{W^*_{n}V_{{n+1}}^*V_{{n+1}} W_n}{\tr(W^*_n W_n)} \, Y \, \tr\big(W^*_n W_n\big) \, \d \mu\pt {n+1} \\
					   &=\frac{1}{k} \int_{M_k(\cc)^n} \frac{W^*_{n} W_n}{\tr\big(W^*_n W_n\big)}\,  Y \, \tr\big(W^*_n W_n\big)\, \d \mu\pt n\\
					   &= \eech(Y M_n),
\end{align*}
where the second equality follows from the stochasticity condition \eqref{eq:stochastic family}, $\int v^* v d \mu(v) = \id_{\cc_k}$. This shows that $(M_n)$ is a $(\outalg_n)$-martingale w.r.t.\ $\ppch$. Since the sequence $(M_n)$ is composed of positive semidefinite matrices of trace one, its coordinates are a.s.\ uniformly bounded by $1$. Therefore, the martingale property implies the $L^1$ and a.s. convergence of $(M_n)$ to a $\outalg$-measurable random variable $M_\infty$. Now note that for any $\rho\in\mathcal D_k$,
$$\tr(W_n^* W_n\rho)=\tr(M_n\rho)\,k\,\tr\left(W_n^* W_n \,\frac{\id_{\mathbb C_k}}{k}\right).$$
This way, the convergence of $(M_n)$ implies the change of measure formula.

We now prove the last part of the Proposition. Using the martingale property one can see that for all $n\in\mathbb N$, and any fixed $p \in \mathbb{N}$,
\begin{eqnarray}
V_n^p\ =\ \sum_{k=0}^{p-1}\eech(M_{k+n+1}^2 -  M_k^2)
&=&\sum_{k=0}^n\eech(M_{k+p}^2)-\sum_{k=0}^n\eech(M_k^2)\nonumber\\
&=&\sum_{k=0}^n\eech\big((M_{k+p}-M_k)^2\big)\nonumber\\
&=&\eech\Big(\sum_{k=0}^n\eech\big((M_{k+p}-M_k)^2\vert\mathcal O_k\big)\Big).
\end{eqnarray}
Since $(M_n)$ is bounded and almost surely convergent, the Lebesgue dominated convergence Theorem implies that the term $V_n^p$ is convergent when $n$ goes to infinity. Hence we get that 
$$\eech\Big(\sum_{k=0}^\infty\eech\big((M_{k+p}-M_k)^2\vert\mathcal O_k)\big)\Big)<+\infty
$$
which yields that
$$\lim_{n\rightarrow\infty}\eech\big((M_{n+p}-M_n)^2\vert\mathcal O_n\big)=0.$$
At this stage we use the polar decomposition of $(W_n)$, Eq.~\eqref{eq:polar}, to write
$$
M_{n+p} = \frac{M_n^{\frac12}U_n^*V_{n+1}^*\ldots V_{n+p}^*V_{n+p}\ldots V_{n+1}U_nM_n^{\frac12}}{\tr(M_n^{\frac12}U_n^*V_{n+1}^*\ldots V_{n+p}^*V_{n+p}\ldots V_{n+1}U_nM_n^{\frac12})}.
$$
Then we get an expression for the conditional expectation, c.f. the first part of the proof,
\begin{eqnarray*}
\eech\big((M_{n+p}-M_n)^2\vert\mathcal O_n\big)&=&\int_{M_k(\mathbb C)^p}\left(\frac{M_n^{\frac12}U_n^*v_{1}^*\ldots v_{p}^*v_{p}\ldots v_{1}U_nM_n^{\frac12}}{\tr(M_n^{\frac12}U_n^*v_{1}^*\ldots v_{p}^*v_{p}\ldots v_{1}U_nM_n^{\frac12})}-M_n\right)^2\\&&\qquad\qquad\qquad\qquad\tr(v_{1}^*\ldots v_{p}^*v_{p}\ldots v_{1}U_nM_nU_n^*)\,\d\mu^{\otimes p}(v_1,\ldots,v_p).
\end{eqnarray*}
Since $(U_n)$ are unitary matrices, $\ppch\as$, there exists a subsequence  along which $(U_n)$ converges to some random unitary operator $U_\infty$. Taking the limit as $n$ goes to infinity along this subsequence in the above expression, yields
\begin{eqnarray*}\lefteqn{\int_{M_k(\mathbb C)^p}\left(\frac{M_\infty^{\frac12}U_\infty^*v_{1}^*\ldots v_{p}^*v_{p}\ldots v_{1}U_\infty M_\infty^{\frac12}}{\tr(M_\infty ^{\frac12}U_\infty^*v_{1}^*\ldots v_{p}^*v_{p}\ldots v_{1}U_\infty M_\infty^{\frac12})}-M_\infty\right)^2}\\&\qquad\qquad\qquad\tr(v_{1}^*\ldots v_{p}^*v_{p}\ldots v_{1}U_\infty M_\infty U_\infty^*)\,\d\mu^{\otimes p}(v_1,\ldots,v_p)=0,\quad \ppch\as
\end{eqnarray*}
This implies that 
$$\left(\frac{M_\infty^{\frac12}U_\infty^*v_{1}^*\ldots v_{p}^*v_{p}\ldots v_{1}U_\infty M_\infty^{\frac12}}{\tr(M_\infty ^{\frac12}U_\infty^*v_{1}^*\ldots v_{p}^*v_{p}\ldots v_{1}U_\infty M_\infty^{\frac12})}-M_\infty\right)^2\tr(v_{1}^*\ldots v_{p}^*v_{p}\ldots v_{1}U_\infty M_\infty U_\infty^*)=0,\,\,$$
$\ppch$-almost surely and for $\mu^{\otimes p}$-almost all $(v_1,\ldots,v_p)$. Note that the notion ``$\ppch$-almost surely'' stands for $M_\infty$ and $U_\infty$ whereas ``$\mu^{\otimes p}$-almost all'' stands for the $v_i's$. The product is zero, only if at least one of the terms in the product is zero. It turns out, however, that both cases can be described by a single condition: the product vanishes only if there exists $\lambda$ such that
$$M_\infty^{\frac12}U_\infty^*v_{1}^*\ldots v_{p}^*v_{p}\ldots v_{1}U_\infty M_\infty^{\frac12}=\lambda M_\infty.$$
Denoting by $\pi_\infty$ the orthogonal projector onto the range of $M_\infty$, the condition is equivalent to
$\pi_\infty U_\infty^*v_{1}^*\ldots v_{p}^*v_{p}\ldots v_{1}U_\infty \pi_\infty=\lambda \pi_\infty$.
Finally, it follows that
$$U_\infty \pi_\infty U_\infty^*v_{1}^*\ldots v_{p}^*v_{p}\ldots v_{1}U_\infty \pi_\infty U_\infty^*\propto U_\infty \pi_\infty U_\infty^*,$$
for  $\mu^{\otimes p}$-almost all $(v_1,\ldots,v_p)$. Since $U_\infty \pi_\infty U_\infty^*$ is an orthogonal projector, the condition \pur\ implies that $\rank(M_\infty)=\rank(U_\infty \pi_\infty U_\infty^*)=1$. Since $M_\infty$ is a trace one, positive semidefinite matrix this means that $M_\infty$ is a rank one projector.
\medskip

For the converse implication, assume that $M_\infty$ is $\ppch$-almost surely a rank one projection but that {\pur} does not hold. Then there exists $\proj$, a rank two orthogonal projector, s.t. for all $n\in\nn$, $$\proj W_n^*W_n\proj\propto\proj,$$
$\mu\pt n$-almost everywhere. Since $\mu^{\otimes n}$-almost everywhere $M_n\propto W_n^*W_n$, we get that
$$\proj M_n \proj\propto \proj,$$
$\mu^{\otimes n}$-almost everywhere. Thus,
$\proj M_\infty \proj\propto \proj,$ and under our assumption that $\rank M_\infty=1$ $\ppch\as$ and $\rank \pi=2$ this implies that $\proj M_\infty \proj=0$, $\ppch$-almost surely. On the other hand for all $n\in\mathbb N$ we have $\eech(M_n)=\id_{\cc^k}$, and the $L^1$ convergence implies that $\eech(M_\infty)=\id_{\cc^k}$. Then, $\eech(\proj M_\infty \proj)=\proj$ which contradicts $\proj M_\infty \proj=0$ $\ppch\as$
\end{proof}

By the polar decomposition, the rank of $W_n$ is equal to the rank of $M_n$ and the proposition thus implies that $W_n \rho_0 W_n^*/\tr(W_n \rho_0 W_n^*)$ approaches the set of pure states for any $\rho_0\in\mathcal D_k$ if and only if \pur\ holds. This is the result of Maassen and K\"ummerer \cite{Maassen} mentioned in the introduction. Though $M_n$ is not used in \cite{Maassen}, the proof relies on similar ideas.
\medskip

We are now in the position to show that the Markov chain $(\hat{x}_n)$ is asymptotically an $\outalg$-measurable process. This is expressed in the following lemma. 
Whenever {\pur} holds, we denote by $\hat{z} \in \sta$ the $\outalg$-measurable random variable defined by
$$
M_\infty = \pi_{\hat z}.
$$
Recall that $\dist(\cdot,\cdot)$, defined by Eq.~\eqref{eq:metric}, is our metric on $\sta$.

\begin{lemma}\label{lemma:polar}
Assume {\pur} holds. Then for any probability measure $\nu$ on $(\sta,\mathcal B)$,
\[\lim_{n\to \infty} d(\hat x_n,U_n \cdot \hat z)=0\quad \mathbb{P}_\nu\as\]
\end{lemma}
\begin{proof}
We start the proof by showing that for any $\nu$
\begin{equation}\label{eq:racinem}
\lim_{n\to\infty} M_n^{\frac12} \cdot \hat x=\hat z\quad \mathbb P_\nu\as
\end{equation}
Let $\hat x$ be fixed and recall from \Cref{lem:conv proj} that \pur\ implies $M_\infty=\pi_{\hat z}$. In order to show \eqref{eq:racinem}, it is enough to show that $\hat x$ is almost surely not orthogonal to $\hat z$. From \Cref{eq:desintegration} and the change of measure formula in \Cref{lem:conv proj},
$$\d\pp_{\nu}(\hat{x}, \omega) = k\,\tr(\proj_{\hat x} \proj_{\hat z(\omega)})\, \d\big(\nu(\hat{x})\otimes \ppch(\omega)\big).$$
Hence the event $\{\dist(\hat x,\hat z)=1\}$ has $\mathbb P_{\nu}$-measure $0$, and \eqref{eq:racinem} follows from the almost sure convergence of $M_n$ to $\proj_{\hat{z}}$.

Now using the polar decomposition, Eq.~\eqref{eq:polar}, and the fact that proportionality of vectors  means equality of their equivalence classes in $\P(\mathbb C^k)$, we have
$$
\hat{x}_n  = U_n M_n^{\frac12}\cdot \hat x_0.
$$
The first part of the proof then yields,
\[\lim_{n\to \infty} d(\hat x_n,U_n \cdot \hat z)=0,\quad \mathbb{P}_\nu \as\]
\end{proof}

The uniqueness of the invariant measure which is the first part of \Cref{thm:uniqueness} follows as a corollary.

\begin{corollary} \label{coro_uniqueness}
Assume {\pur} and {\irr}. Then the Markov kernel $\Pi$ admits a unique invariant probability measure.
\end{corollary}
\begin{proof}
For an invariant measure $\nu$, the random variable $\hat{x}_n$ is $\nu$-distributed for all $n \in \mathbb{N}$. In particular, $
\mathbb{E}_\nu\big(f(\hat{x}_n)\big)$ is constant for any continuous function $f$. On the other hand \Cref{lemma:polar} and Lebesgue's dominated convergence theorem imply that 
$$\lim_{n\to \infty} \mathbb E_{\nu}\big(f(\hat x_n)-f(U_n\cdot\hat z)\big) =0.$$
Hence 
\begin{equation}
\label{eq:coro1}
\lim_{n\to \infty} \mathbb E_{\nu}\big( f(U_n\cdot\hat z)\big)  = \mathbb{E}_{\nu} \big(f(\hat{x}_0)\big).
\end{equation}

Assume now that there exist two invariant measures $\nu_a$ and~$\nu_b$. Since $U_n\cdot\hat z$ is $\outalg$-measurable, \Cref{prop:marginal} implies
$$\mathbb E_{\nu_a}\big(f(U_n\cdot\hat z)\big)=\mathbb E_{\nu_b}\big(f(U_n\cdot\hat z)\big).$$
Then \Cref{eq:coro1} applied with $\nu=\nu_a$, resp. $\nu = \nu_b$ gives
$$\mathbb E_{\nu_a}\big(f(\hat x_0)\big)=\mathbb E_{\nu_b}\big(f(\hat x_0)\big)$$
which means that $\nu_a=\nu_b$ and the uniqueness is proved.
\end{proof}

Assuming only \pur\ we can actually completely characterize the set of invariant measures.
\begin{proposition}
\label{rem:decomposition}
Assuming \pur\ there exists a set $\{F_j\}_{j=1}^d$ of mutually orthogonal subspaces of $\cc^k$ such that,
for each $j\in\{1,\ldots,d\}$ there exists a unique $\Pi$-invariant probability measure $\nu_j$ supported on $\P(F_j)$, and the set of $\Pi$-invariant probability measures is the convex hull of $\{\nu_j\}_{j=1}^d$.
\end{proposition}
The subspaces $F_j$ are the ranges of the extremal fixed points of $\phi$ in $\mathcal D_k$. This is shown in the proof of this Proposition, that is provided in \Cref{app:pur_decompo_unique}.

\begin{remark}
Assuming {\irr}\ only, the chain might or might not have a unique invariant probability measure. Indeed if $\supp\mu\subset SU(k)$ Assumption \pur\ is trivially not verified and, as proved in \Cref{app:SU(k)}, the uniqueness of the invariant measure depends on the smallest closed subgroup of $SU(k)$ containing $\supp\mu$. To illustrate this point, in the same appendix, we study two examples with $\mu$ supported on and giving equiprobability to two elements of $SU(2)$ such that \irr\ holds. In the first example $\Pi$ has a unique invariant probability measure whereas in the second example $\Pi$ has uncountably many mutually singular invariant probability measures.
\end{remark}

\section{Convergence}
\label{sec:speed}
We now turn to the proof of the second part of \Cref{thm:uniqueness}, namely the geometric convergence in Wasserstein distance of the process $(\hat x_n)$ towards the invariant measure. We first recall a definition of this distance for compact metric spaces: for $X$ a compact metric space equipped with its Borel $\sigma$-algebra, the Wasserstein distance of order $1$ between two probability measures $\sigma$ and $\tau$ on $X$ can be defined using Kantorovich--Rubinstein duality Theorem as
$$W_1(\sigma,\tau)=\sup_{f\in \Lip_1(X)}\left| \int_{X} f\,\d\sigma- \int _X f\, \d\tau\right|,$$
where $\Lip_1(X)=\{f:X\rightarrow\mathbb R \ \mathrm{s.t.}\ \vert f(x)-f(y)\vert\leq \dist(x,y)\}$ is the set of Lipschitz continuous functions with constant one, and $\dist$ is the metric on $X$.

The proof of \Cref{eq:conv with period} consists of three parts. In the first part we show a geometric convergence in total variation of $\pp^\rho$ to $\pp^{\rho\inv}$ under the shift $\theta(v_1,v_2,\ldots)=(v_2,v_3,\ldots)$. In the second one we show a geometric convergence of the chain $(\hat x_n)$ towards an $\outalg$-measurable process $(\hat y_n)$. Finally, we combine these results to prove \Cref{eq:conv with period}. 

\subsection{Convergence for $\outalg$-measurable random variables}
Let us first discuss the origin of the integer $m$ in \Cref{eq:conv with period}. Let $(E_1,\ldots,E_\ell)$ be an orthogonal partition of a $\supp \mu$-invariant subspace, i.e.\  a family of mutually orthogonal subspaces such that $E_1\oplus\ldots\oplus E_\ell$ is a $\supp \mu$-invariant subspace. We say that $(E_1,\ldots,E_\ell)$ is a $\ell$-cycle of $\phi$ if $vE_j\subset E_{j+1}$ for $\mu$-a.a. $v$ (with the convention $E_{\ell+1}=E_1$).\footnote{As suggested by its name, the notion of cycle for $\phi$  depends only on $\phi$ and not on the specific measure $\mu$ leading to $\phi$ \cite{wolftour,schrader,CarbonePautrat}.} The set of $\ell\in\nn$ for which there exists an orthogonal partition $(E_1,\ldots,E_\ell)$ is non-empty (as it contains $1$) and bounded (as necessarily $\ell\leq k$). 

\begin{definition} \label{def_period}
	The largest $\ell\in\nn$ such that there exists a $\ell$-cycle of $\phi$ is called the period of $\phi$. We denote this period by $m$.
\end{definition}
\begin{remark} \label{remark_period} \,\hfill
\begin{itemize}
	\item The above definition of period for $\phi$ is similar to that of the period of a $\varphi$-irreducible Markov chain. It is obvious that if $(E_1,\ldots,E_\ell)$ is an $\ell$-cycle of $\phi$ then it is also an $\ell$-cycle of $\Pi$. However, the Markov chain defined by $\Pi$ is not $\varphi$-irreducible in general. Hence the results of \cite{MT} on the period of $\varphi$-irreducible Markov chains do not apply and the characterization of the period of $\Pi$ remains an open problem.
	\item The above definition shows that the union $\bigcup_{j=1}^m E_j$ is invariant by $\mu$-a.a. $v$. Hence, the strong irreducibility assumption discussed at the end of the introduction implies that $m=1$.
\end{itemize}	
\end{remark}
The following result is a reformulation of the Perron--Frobenius theorem of Evans and H\o egh-Krohn, \cite{EHK}. The original formulation in \cite{EHK} makes the additional assumption that $E=\cc^k$ in \irr. For the present extension see e.g.\ \cite{wolftour}. In the following statement, and in the rest of the article, for $X$ an operator on $\cc^k$ we denote $\|X\|_1=\tr |X|$ (all statements are identical with a different norm, but this choice will spare us a few irrelevant constants).
\begin{theorem}\label{thm:ergod}
	Assume \irr\ holds. Then there exists a unique $\phi$-invariant element $\rho\inv$ of $\mathcal D_k$ with range equal to the minimal invariant subspace $E$. In addition, there exist two positive constants $c$ and $\lambda<1$ such that, with $m$ defined in \Cref{def_period}, for any $\rho\in\mathcal D_k$ and for any $n\in\nn$,
	\begin{equation}\label{eq:ergod}
	\left\|\frac1m\sum_{r=0}^{m-1} \phi^{mn+r}(\rho)-\rho\inv\right\|_1\leq c\lambda^n.
	\end{equation}
\end{theorem}
\begin{proof}
Theorem 4.2 in \cite{EHK} implies that $\rho\inv$ is the unique $\phi$-invariant element of $\mathcal D_k$, that the eigenvalues of $\phi$ of modulus one are exactly the $m$-th roots of unity, and that they are all simple. The statement follows, with $\lambda$ the modulus of the largest non-peripheral eigenvalue. 
\end{proof}

Recall that $\theta$ is the left shift operator on~ $\Omega$, i.e.\  
$$\theta(v_1,v_2,\ldots)=(v_2,v_3,\ldots).$$
The main result of this section is the following proposition. As announced it concerns the speed of convergence in total variation (expressed in terms of expectation values).
\begin{proposition} \label{prop_cvgFmeasurable}
	Assume \irr\ holds. Then there exist two positive constants $C$ and $\lambda<1$ such that for any $\outalg$-measurable function $f$ with essential bound $\|f\|_\infty$, any $\rho\in\mathcal D_k$ and all $n\in\nn$,
	\begin{equation}
	\left\vert \ee^\rho\left(\frac{1}{m}\sum_{r=0}^{m-1}f\circ\theta^{mn+r}\right)-\ee^{\rho\inv}(f)\right\vert\leq  C\|f\|_\infty\lambda^n.
	\end{equation}
\end{proposition}
\begin{proof}

We claim that for any bounded $\outalg$-measurable function $f$,
\begin{equation}
\label{eq:prop_cvg_1}
\ee^\rho(f \circ \theta) = \ee^{\phi(\rho)}(f).
\end{equation}
It suffices to prove the relation for all $\outalg_l$-measurable functions for some integer $l$. For such a function,
\begin{align*}
\ee^\rho(f\circ\theta)=&\int_{M_k(\cc)^{l+1}} f(v_{2},\ldots,v_{l+1})\tr(v_{l+1}\ldots v_{1}\rho v_1^*\ldots v_{l+1}^*)\, \d\mu\pt{(l+1)}(v_1,\ldots,v_{l+1})\\
	=&\int_{M_k(\cc)^{l}} f(v_{1},\ldots,v_{l})\tr(v_{l}\ldots v_{1}\phi(\rho)v_{1}^*\ldots v_{l}^*)\, \d\mu\pt{l}(v_{1},\ldots,v_{l}),
\end{align*}
which is equal to $\ee^{\phi(\rho)}(f)$ as claimed.

Applying \Cref{eq:prop_cvg_1} multiple times and using the change of measure of \Cref{lem:conv proj} we obtain 
\begin{align*}
\ee^\rho\left(\frac1m\sum_{r=0}^{m-1} f\circ\theta^{mn+r}\right) &= \frac{1}{m} \sum_{r=0}^{m-1} \ee^{\phi^{nm+r} (\rho)}(f) \\					
&=  k \frac{1}{m} \sum_{r=0}^{m-1} \eech\big(f\,  \tr\big( M_\infty \phi^{nm+r} (\rho)\big)\big),
\end{align*}
for any $\outalg$-measurable function $f$.
Using $|\tr(M_\infty A)|\leq \|A\|_1$ for $A=A^*$ (remark that $M_\infty\in\mathcal D_k$ by construction) we then obtain
\begin{align*}
\left|\ee^\rho\left(\frac1m\sum_{r=0}^{m-1} f\circ\theta^{mn+r}\right)-\ee^{\rho\inv}(f)\right|\leq \|f\|_\infty\, k\left\|\frac1m\sum_{r=0}^{m-1}\phi^{mn+r}(\rho)-\rho\inv\right\|_1
\end{align*}
and \Cref{thm:ergod} yields the proposition with $C=ck$.
\end{proof}

\subsection{Convergence to a ${\outalg}$-measurable process} \label{subsec_cvgFmeas}

Let us introduce two relevant processes: 
for all $n\in\nn$, let
\begin{equation} \label{eq_defyn}
\hat z_{n}(\omega)=\mathop{\mathrm{argmax}}_{\hat x\in \sta}\,\|W_n x\|^2
\end{equation}
and
\begin{equation} \label{eq_defzn}
\hat y_n=W_n\cdot\hat z_n.
\end{equation}
Both random variables $\hat y_n$ and $\hat z_n$ are $\outalg_n$-measurable.

The random variable $\hat z_n$ corresponds to the maximum likelihood estimator of $\hat x_0$. Note that the $\operatorname{argmax}$ may not be uniquely defined. We can, however, define it in an  $\outalg_n$-measurable way. The following results will not be affected by such a consideration, and we will not discuss such questions in the sequel. It follows from the definition of $\hat{z}_n$ that
\begin{equation}
\label{eq:svd2}
(W_n^*W_n)^{\frac12}\, z_n=\|W_n\| z_n, \quad \ppch\as.
\end{equation} 
We recall that $z_n$ is a vector representative of the class $\hat{z}_n$.

Concerning $\hat y_n$, it can be seen as an estimator of $\hat x_n$ given the maximum likelihood estimation of $\hat x_0$. The following proposition establishes consistency of this estimator, we show the geometric contraction in the mean of $(\hat x_n)$ and $(\hat y_n)$. In fact we prove a slightly more general statement that the estimator based on the first $n$ outcomes can be replaced by an estimator based on outcomes in between $l$ and $l+n$.  We will prove the almost sure contraction in \Cref{prop:expo_conv dist}.
\begin{proposition}\label{prop:mean exp conv}
Assume \pur\ holds. Then there exist two positive constants $C$ and $\lambda<1$ such that for any probability measure $\nu$ over $(\sta,\mathcal B)$,
\begin{equation}\label{eq:unifink}
\ee_{\nu}\big(d(\hat x_{n+l}, \hat y_n\circ \theta^l)\big)\leq C\lambda^n,\end{equation}
holds for all non-negative integers $l$ and $n$.
\end{proposition}

In order to prove \Cref{prop:mean exp conv} we study the largest two singular values of $W_n$. As is customary in the study of products of random matrices, we make use of exterior products. We recall briefly the relevant definitions: for $p\in\nn$ and $p$ vectors  $x_1, \ldots, x_p$ in $\cc^k$ we denote by $x_1\wedge \ldots \wedge x_p$ the alternating bilinear form $(y_1,\ldots, y_p)\mapsto \det\big(\langle x_i, y_j\rangle \big)_{i,j=1}^p$. Then, the set of all $x_1\wedge \ldots \wedge x_p$ is a generating family for the set $\wedge^p\cc^k$ of alternating bilinear forms on $\cc^k$, and we can define a hermitian inner product by 
\[\langle x_1\wedge \ldots \wedge x_p, y_1\wedge \ldots \wedge y_p\rangle = \det\big(\langle x_i, y_j\rangle \big)_{i,j=1}^p, \]
and denote by $\|x_1\wedge \ldots \wedge x_p\|$ the associated norm. It is immediate to verify that our metric $\dist$, defined by \eqref{eq:metric} satisfies
\begin{equation}\label{eq:d wedge}
\dist(\hat x,\hat y)=\frac{\|x\wedge y\|}{\|x\|\|y\|}.
\end{equation}
For $A$ an operator on $\mathbb C^k$, we write $\wedge^p X$ for the operator on $\wedge^p\cc^k$ defined by
\begin{equation}\label{eq_defwedgepA}
\wedge^p A \,(x_1\wedge \ldots \wedge x_p)=Ax_1\wedge \ldots \wedge Ax_p.
\end{equation}
Obviously $\wedge^p (AB)=\wedge^p A\wedge^p B$, so that $\|\wedge^p (AB)\|\leq\|\wedge^p A\|\|\wedge^p B\|$. From e.g.\  Chapter XVI of \cite{BMLalgebra} or Lemma III.5.3 of \cite{BouLac}, we have in addition for $1\leq p\leq k$
\begin{equation}\label{eq:wedge_singular val}
\|\wedge^p A\|=a_1(A)\ldots a_p(A),
\end{equation}
where $a_1(A)\geq \ldots \geq a_k(A)$ are the singular values of $A$, i.e.\  the square roots of eigenvalues of $A^* A$, labelled in decreasing order.

Our strategy to prove \Cref{prop:mean exp conv} is to bound the right hand side of \Cref{eq:unifink} by a submultiplicative function $f : \mathbb{N} \to \mathbb{R}_+$ and then use the Fekete's lemma. We will show that the function 
\begin{equation} \label{eq_deffn}
f(n)=\int_{M_k(\cc)^n}  \|\wedge^2 v_n\ldots v_1\|\,\d\mu\pt n(v_1,\ldots,v_n)
\end{equation}
have these desired properties. The following lemma establishes an exponential decay of this function.

\begin{lemma} \label{lemma_submult}
Assume {\pur}. Then there exist two positive constants $C$ and $\lambda<1$ such that
\[f(n)\leq C\lambda^n.\]
\end{lemma}

\begin{proof}
First, we prove $\lim_{n\to\infty}f(n)=0$. To prove it, we express the function $f(n)$ using the process $W_n$ as 
\begin{equation} \label{eq_altdeffn}
f(n)=\eech\left(k\frac{\Vert \wedge^2 W_n\Vert}{\tr(W_n^*W_n)}\right).
\end{equation}
By definition the eigenvalues of $M_n^{\frac12}$ are the singular values of $W_n/\sqrt{\tr(W_n^*W_n)}$. Since by Proposition~\ref{lem:conv proj}, $M_n$ converges $\ppch\as$ to a rank one projection,
$$\lim_{n\to\infty}a_1\left(\frac{W_n}{\sqrt{\tr(W_n^*W_n)}}\right)a_2\left(\frac{W_n}{\sqrt{\tr(W_n^*W_n)}}\right)=0\quad \ppch\as$$
Using \Cref{eq:wedge_singular val} we then conclude that 
\begin{equation}
\lim_{n\mapsto\infty}\frac{\Vert \wedge^2W_n\Vert}{\tr(W_n^*W_n)}=0 \quad \ppch\as
\end{equation}
Since  $\|\wedge^2 W_n\|\leq \|W_n\|^2\leq \tr(W_n^*W_n)$, the expression \eqref{eq_altdeffn} and  Lebesgue's dominated convergence theorem imply $\lim_{n\to\infty}f(n)=0$.

\smallskip
Second,  remark that the function $f$ is submultiplicative. Indeed, for $p,q\in\mathbb N$ we have
\[
\|\wedge^2(v_{p+q}\ldots v_{1})\| \leq \|\wedge^2(v_{p+q}\ldots v_{p+1})\|  \|\wedge^2(v_{p}\ldots v_{1})\|
\]
and the submultiplicativity follows.

By Fekete's subadditive Lemma, $\frac{\log f(n)}{n}$ converges to $\inf_{n\in\nn} \frac{\log f(n)}{n}$, which is (strictly) negative (and possibly equal to $-\infty$) since $f(n)\to 0$. Then there exists $0<\lambda<1$ such that $f(n)\leq \lambda^n$ for large enough $n$, and the conclusion follows.
\end{proof}

We are now in position to prove \Cref{prop:mean exp conv}.

\begin{proof}[Proof of \Cref{prop:mean exp conv}]
The Markov property of $(\hat x_n)$ implies that
$$
\ee_{\nu}\big(d(\hat x_{n+l}, \hat y_n\circ \theta^l)\big) = \ee_{\nu\Pi^l}\big(d(\hat x_{n}, \hat y_n)\big).
$$
Provided inequality \eqref{eq:unifink} is established for $l=0$, the right hand side of the previous equality can be bounded by $C \lambda^n$. It is hence sufficient to prove the inequality for $l=0$.

The case $l=0$ follows from \Cref{lemma_submult} if for any $n \in \nn$ and any probability measure $\nu$,
\begin{equation} \label{eq_fundineq}
\mathbb E_{\nu}\big(d(\hat x_n, \hat y_n)\big)\leq f(n).
\end{equation}
To obtain this inequality, note that from the definitions of $\hat x_n$, $\hat y_n$ and $\hat z_n$, we have that
\begin{align*}
\dist(\hat x_n, \hat y_n)&=\frac{\|\wedge^2 W_n\,(x_0\wedge z_n)\|}{\|W_n x_0\|\|W_n z_n\|}\\
&\leq \frac{\|\wedge^2 W_n\|}{\|W_n x_0\|^2}\frac{\|W_n x_0\|}{\|W_n\|}\\
&\leq \frac{\|\wedge^2 W_n\|}{\|W_n x_0\|^2},
\end{align*}
holds $\pp_\nu$-almost surely. To get the first inequality we used $\|W_n z_n\|= \|W_n\|$, and $\|x_0\wedge z_n \|=\dist(\hat x_0,\hat z_n)\leq 1$. In addition, by definition of $\pp_\nu$,
\begin{align*}
\ee_{\nu}\left(\frac{\|\wedge^2 W_n\|}{\|W_n x_0\|^2}\right)&= \int_{\P(\cc^k)\times M_k(\cc)^n}
\frac{\|\wedge^2W_n\|}{\|W_n x_0\|^2}\,\|W_n x_0\|^2\,\d\mu\pt n\, \d\nu(\hat x_0) \\
&= \int_{M_k(\cc)^n}
{\|\wedge^2W_n\|}\,\d\mu\pt n(v_1,\ldots, v_n),
\end{align*}
which is $f(n)$. Therefore \eqref{eq_fundineq} holds and \Cref{lemma_submult} yields the proof.
\end{proof}

\subsection{Convergence in Wasserstein metric}
The remainder of \Cref{sec:speed} is directly devoted to the proof of the second part of Theorem~\ref{thm:uniqueness}.
\begin{proof}[Proof of \Cref{eq:conv with period}]
We are supposed to prove that
$$ W_1\Big(\frac1m\sum_{r=0}^{m-1} \nu\Pi^{mn+r}, \nu\inv\Big) =  \sup_{f\in \Lip_1(\P(\mathbb C^k))} \left| \mathbb{E}_\nu\left( \frac1m\sum_{r=0}^{m-1} f(\hat{x}_{mn+r})\right) - \mathbb{E}_{\nu_{\inv}}(f(\hat{x}_0)) \right|  $$
is exponentially decaying in $n$. The expression in the supremum on the right hand side is unchanged by adding an arbitrary constant to $f$. This freedom allows us to restrict the supremum to functions bounded by $1$, i.e. $\|f\|_\infty \leq 1$.

Let $f\in \Lip_1(\P(\mathbb C^k))$ be such a function. Our strategy is to approximate $\hat{x}_{mn+r}$ by $\hat{y}_{mp} \circ \theta^{mq+r}$ with $p=\lfloor \frac n2 \rfloor$ and $q=\lceil \frac n2\rceil$ so that in particular $p+q =n$. Using telescopic estimates and the invariance of $\nu\inv$ we then have 
\begin{align*}
\left|\ee_\nu \left(\frac1m \sum_{r=0}^{m-1} f(\hat x_{mn +r})\right) - \ee_{\nu\inv}\big(f(\hat x_0)\big)\right|
&\leq \frac1m \sum_{r=0}^{m-1} \big|\ee_\nu \big(f(\hat x_{m(p+q)+r})\big) - \ee_\nu \big(f(\hat y_{mp} \circ \theta^{mq+r})\big)\big| \\
& + \frac1m \sum_{r=0}^{m-1} \big| \ee_{\nu\inv} \big(f(\hat y_{mp} \circ \theta^{mq+r})\big)-\ee_{\nu\inv} \big(f(\hat x_{m(p+q)+r})\big)\big| \\
&  +  \left| \frac1m \sum_{r=0}^{m-1}\ee_\nu \big(f(\hat y_{mp} \circ \theta^{mq+r})\big) - \ee_{\nu\inv} \big(f(\hat y_{mp}) \big)\right|.
\end{align*}
We bound the terms on the right hand side using  \Cref{prop:mean exp conv} for the first two terms and \Cref{prop_cvgFmeasurable} for the last term. To this end let $C$ and $\lambda < 1$ be such that bounds in both these propositions hold true.
Since $f$ is $1$-Lipschitz continuous we have
$$
|f(\hat x_{m(p+q)+r}) - f(\hat y_{mp} \circ \theta^{mq+r})| \leq \dist(\hat x_{m(p+q)+r},\hat y_{mp} \circ \theta^{mq+r}).
$$
\Cref{prop:mean exp conv} then implies that 
$$
 \big|\ee_\nu \big(f(\hat x_{m(p+q)+r})\big) - \ee_\nu \big(f(\hat y_{mp} \circ \theta^{mq+r})\big)\big| \leq C\lambda^{mp},
$$
and similarly with $\nu$ replaced by $\nu\inv$.
Regarding the last term in the above telescopic estimate we have by \Cref{prop_cvgFmeasurable},
$$
\left| \frac1m \sum_{r=0}^{m-1}\ee_\nu \big(f(\hat y_{mp} \circ \theta^{mq+r})\big) - \ee_{\nu\inv} \big(f(\hat y_{mp}) \big)\right|
\leq C \lambda^{q},
$$
where we used the constraint $\|f\|_\infty\leq 1$ discussed at the beginning of the proof.

Putting these estimates together we get
\[
\left|\ee_\nu \left(\frac1m \sum_{r=0}^{m-1} f(\hat x_{mn+r})\right) - \ee_{\nu\inv}\big(f(\hat x_0)\big)\right|\leq 3C\lambda^{\lfloor \frac n2 \rfloor}
\]
and this concludes the proof of \Cref{eq:conv with period} and therefore of \Cref{thm:uniqueness}.
\end{proof}

\section{Lyapunov exponents} \label{sec_Lyapunoventropy}
In this section, we study the almost sure stability exponents. We will always assume \irr\ with the additional assumption that the unique minimal $\supp\mu$-invariant subspace $E$ is $\cc^k$.

\begin{remark}
Assuming $E=\cc^k$ amounts to saying that $\phi$ has no transient part. Without this assumption, we would have to take into account the almost sure Lyapunov exponent corresponding to the escape from the transient part. See \cite{BPT} for a precise account of these ideas. 
\end{remark}

The relevance of this assumption will stem from the following straightforward inequalities: if $\rho$ is any element of $\mathcal D_k$ then one has
\begin{equation*} \label{eq_absco1}
\frac{\d \pp^\rho_{\vert\mathcal O_n}}{\d \mu^{\otimes n}} \leq \|W_n\|^2,
\end{equation*}
%with $\pp^\rho_n$ the projection of $\pp^\rho$ onto $\outalg_n$ 
and if $\rho$ is faithful (i.e.\ definite positive), then 
\begin{equation*} \label{eq_absco2}
\frac{\d\pp^\rho_{\vert\mathcal O_n}}{\d \mu^{\otimes n}} \geq \|\rho^{-1}\|^{-1}\|W_n\|^2.
\end{equation*}
In particular, under the assumption that  \irr\ holds with $E=\cc^k$, thus $\rho\inv>0$ and for any $\rho\in\mathcal D_k$, we have 
\begin{equation}\label{eq:E=cck_ac}
\pp^\rho\ll\pp^{\rho\inv}.
\end{equation}

Let us start by proving the following lemma which concerns ergodicity of $\theta$ w.r.t.\ the measure $\pp^{\rho\inv}$.
\begin{lemma}\label{lem:ppinv_ergodic}
Assume that \irr\ holds. Then the shift $\theta$ on $(\out,\outalg)$ is ergodic with respect to the probability measure $\pp^{\rho\inv}$.
\end{lemma}
\begin{proof}
Let $A$, $A'$ in $\outalg_\ell$. From the definition of $\pp^{\rho\inv}$, for $j$ large enough, $\pp^{\rho\inv}\big(A \cap \theta^{-j}(A')\big)$ equals
\begin{align*}
\int_{A\times A'} \tr\Big(v_{\ell}'\ldots v_{1}' \phi^{j-l}\big(v_{\ell}\ldots v_{1} \rho\inv v_{1}^*\ldots v_{\ell}^*\big){v_{1}'}^*\ldots {v_{\ell}'}^* \Big) \,\d\mu\pt\ell(v_1,\ldots,v_\ell)\, \d\mu\pt\ell(v_1',\ldots,v_\ell'),
\end{align*} 
and the Perron--Frobenius \Cref{thm:ergod} implies
\[\lim_{n\to\infty}\frac1n \sum_{j=0}^{n-1}  \phi^j\big(v_{\ell}\ldots v_{1} \rho\inv v_{1}^*\ldots v_{\ell}^*\big) = \tr \big(v_{\ell}\ldots v_{1} \rho\inv v_{1}^*\ldots v_{\ell}^*\big)\, \rho\inv\]
for $\mu^{\otimes l}$-almost all $(v_1,\ldots,v_l)$
so that 
\[\lim_{n\to\infty}\frac1n \sum_{j=0}^{n-1} \pp^{\rho\inv}\big(A \cap \theta^{-j}(A')\big) =\pp^{\rho\inv}(A)\,\pp^{\rho\inv}(A'),\]
which proves the ergodicity.
\end{proof}

Now we can state our result concerning Lyapunov exponents.

\begin{proposition}\label{prop:gamma_p}
Assume that {\irr}\ holds with $E=\cc^k$, and that {\pur}\ holds. Assume $\int\|v\|^2\log \|v\|^2 \d\mu(v)< \infty$. Then there exist numbers $$\infty>\gamma_1\geq \gamma_2\geq \cdots\geq\gamma_k\geq -\infty$$ such that for any probability measure $\nu$ over $(\P(\mathbb C^k),\mathcal B)$:
\begin{enumerate}[label=(\arabic*)]
\item\label{it:lyap1} for any $p\in \{1,\ldots,k\}$,
\begin{equation} \label{eq_defgamma}
\lim_{n\to\infty}\frac1n\log\|\wedge^p W_n\|=\sum_{j=1}^p \gamma_j,\quad \pp_\nu\as,
\end{equation}
\item\label{it:lyap2} $\gamma_2-\gamma_1<0$ with $\gamma_2-\gamma_1$ understood as the limit of $\frac1n\log\frac{\|\wedge^2 W_n\|}{\|W_n\|^2}$ whenever $\gamma_1=-\infty$,
\item \label{it:lyap3}
\begin{equation}\label{eq:gamma_1}
\lim_{n\to\infty} \frac1n(\log\|W_n  x_0\|-\log\|W_n\|)=0\quad \pp_\nu\as
\end{equation}
\end{enumerate}
\end{proposition}

\begin{proof}
 Let us start by proving \ref{it:lyap1}. Note that $n\mapsto \log\|\wedge^p W_n\|$ is subadditive by definition. The existence of the $\pp^{\rho\inv}\as$ limits $\lim_{n\to\infty}\frac1n\log\|\wedge^p W_n\|$ then follows from $\ee^{\rho\inv}(\log\|V\|^2)\leq \int\|v\|^2\log \|v\|^2d\mu(v)<\infty$, $\pp^{\rho\inv}\circ\theta^{-1}=\pp^{\rho\inv}$ and a direct application of Kingman's subadditive ergodic Theorem (see e.g.\ \cite{walters}). The fact that these limits are $\pp^{\rho\inv}\as$ constant comes from the $\theta$-ergodicity of $\pp^{\rho\inv}$ proved in \Cref{lem:ppinv_ergodic}. Since by \Cref{eq:E=cck_ac} any $\pp^\rho$ is absolutely continuous with respect to $\pp^{\rho\inv}$,
 \Cref{prop:marginal} and the $\outalg$-measurability of $\|\wedge^pW_n\|$ imply the convergence holds $\pp_\nu\as$ The numbers $\gamma_j$ are uniquely defined, by defining $\sum_{j=1}^p \gamma_j$ as the $\pp^{\rho\inv}\as$ limit $\lim_{n\to\infty}\frac1n\log\|\wedge^p W_n\|$ and imposing the rule that $\gamma_{j+1}=-\infty$ if $\gamma_{j}=-\infty$. This convention and \eqref{eq:wedge_singular val} impose that $\gamma_{j+1}\leq \gamma_j$ for $j=1,\ldots,k-1$.

 Concerning \ref{it:lyap2}, recall the quantity $f(n)$ defined in \Cref{eq_deffn}. Then \Cref{eq_altdeffn} and the inequality $\tr\,W_n^*W_n \leq k \|W_n\|^2$ give
\[f(n)\geq \eech\left(\frac{\|\wedge^2 W_n\|}{\|W_n\|^2} \right). \]
Jensen's inequality implies
\[\frac1n \log f(n)\geq \eech\left(\frac1n \log\frac{\|\wedge^2 W_n\|}{\|W_n\|^2} \right) \]
so that by \Cref{lemma_submult} and Fatou's lemma, $\log \lambda \geq \gamma_2-\gamma_1$ with $\lambda\in(0,1)$.

 Finally for \ref{it:lyap3}, from \Cref{lem:conv proj}, we have
\[\lim_{n\to\infty}\frac{\|W_n x_0\|}{\|W_n\|}=\lim_{n\to\infty}\frac{\|M_n^{\frac12}x_0\|}{\|M_{n}^{\frac12}\|}=|\langle x_0, z\rangle|\quad \pp_\nu\as\]
Since $\pp_\nu\as$ $|\langle x_0, z\rangle|>0$ the proposition holds.

\end{proof}

From this proposition we deduce the following almost sure convergence rate for the distance between the Markov chain $(\hat x_n)$ and the $(\outalg_n)$-adapted process $(\hat y_n)$.
\begin{proposition}\label{prop:expo_conv dist}
Assume \pur\ holds and \irr\ holds with $E=\cc^k$. Then for any probability measure $\nu$ on $(\sta,\mathcal B)$,
\[\limsup_{t\to\infty}\frac1n\log\big(d(\hat x_n, \hat y_n)\big)\leq -(\gamma_1-\gamma_2)<0,\quad \pp_\nu\as\]
\end{proposition}
\begin{proof}
Identity \eqref{eq:d wedge} and the definition of $\hat z_n$ imply
\[{d(\hat x_n,\hat y_n)}=\frac{\|\wedge^2W_n\, x_0\wedge z_n\|}{\|W_nx_0\|\|W_n z_n\|} \leq \frac{\|\wedge^2W_n\|}{\|W_nx_0\|\|W_n\|}.\]
\Cref{prop:gamma_p} then yields the result.
\end{proof}

\paragraph{\textbf{Acknowledgments}} 
T.B. and C.P. would like to thank Y. Guivarc'h for his useful comments at an early stage of this work. Y.P. and C.P. would like to thank P. Bougerol for enlightening discussions about random products of matrices. Y.P. and C.P. would like to thank L. Miclo for relevant discussions regarding Markov chains. The research of T.B., Y.P. and C.P. has been supported by the ANR project StoQ ANR-14-CE25-0003-01 and CNRS InFIniTi project MISTEQ.

\appendix
\section{Equivalence of {\bf (Pur)} and contractivity}\label{app:contractivity_iff_pur}
We assume $\supp \mu\subset GL_k(\cc)$. Recall that $T_\mu$  is the smallest closed sub-semigroup of $GL_k(\cc)$ that contains $\operatorname{supp}\mu$. It is said to be contracting if there exists a sequence $(a_n)_{n\in\mathbb N}\subset T_\mu$ such that $\lim_{n\to\infty} a_n/\|a_n\|$ exists and is a rank one matrix.
\begin{proposition}
Assume $\supp\mu\subset GL_k(\mathbb C)$ and $T_\mu$ is strongly irreducible. Then $\mu$ verifies \pur\ iff. $T_\mu$ is contracting.
\end{proposition}
\begin{proof}
By \Cref{lem:conv proj} the implication {\pur} $\implies$ contractivity follows by taking $(a_n)$ a convergent subsequence of $(W_n(\omega))$ for $\omega\in\operatorname{supp}\ppch$.

We prove the opposite implication by contradiction. Following %\cite[Proposition III.4.4]{BouLac}
\cite[Lemma 3]{GuiRau86}, under the assumptions of the proposition, $T_\mu$ is contracting iff. for any two $\hat x, \hat y\in\sta$ there exists a sequence of matrices $(a_n)\subset T_\mu$ such that
$$\lim_{n\to\infty}d(a_n\cdot \hat x,a_n\cdot \hat y)=0.$$
Now, assume that contractivity holds but {\pur} does not. Namely, that $T_\mu$ is contracting but there exists an orthogonal projector $\proj$ of rank $\geq2$, such that for any $a\in T_\mu$,
$$\proj a^*a\proj\propto \proj.$$
Let $x,y\in\operatorname{range}\pi$ be orthonormal vectors. Then $\langle ax,ay\rangle=\langle x,y\rangle=0$, and $\|ax\|, \|ay\|$ are nonzero, so that $\dist(a\cdot \hat x,a\cdot \hat y)=1$. As this holds for any $a\in T_{\mu}$, contractivity cannot hold. This contradiction yields the proposition.
\end{proof}

\section{Set of invariant measures under assumption {\bf (Pur)}}\label{app:pur_decompo_unique}
A quantum channel is a map $\phi$ on $M_k(\cc)$ of a form
$$
\phi(\rho) = \int_{M_k(\mathcal{C})} v \rho v^* \d \mu(v),
$$
where $\mu$ is a measure satisfying the normalization condition (\ref{eq:stochastic family}). The decomposition of quantum channels to irreducible components was derived in \cite{wolftour,CarbonePautrat,Baumgartner}. The space $\mathbb{C}^k$ is decomposed into orthogonal subspaces, one subspace is transient and in all other subspaces the map has a canonical tensor product structure.  We recall these results.

There exists a decomposition
$$
\mathbb{C}^k  \simeq \mathbb{C}^{n_1} \oplus \dots \oplus \mathbb{C}^{n_d} \oplus \mathbb{C}^{D}, \quad k = n_1 + \dots + n_d + D
$$
with the following properties. We denote by $v_j$ the restriction of $v$ to $\mathbb{C}^{n_j}$. 
\begin{enumerate}[label=(e\arabic*)]
\item\label{it:e1} All invariant states are supported in the subspace $L = \mathbb{C}^{n_1} \oplus \dots \oplus \mathbb{C}^{n_d} \oplus 0$,
\item\label{it:e2} The restriction of $v$ to this subspace is block diagonal,
\begin{equation}
\label{eq:e2}
v|_L = v_1 \oplus \dots \oplus v_d \quad \mu\ae
\end{equation}
\item\label{it:e3} For each  $j=1, \dots ,d$ there is a decomposition $\mathbb{C}^{n_j} = \mathbb{C}^{k_j} \otimes \mathbb{C}^{m_j}, \, n_j = k_j m_j$, a unitary matrix $U_j$ on $\mathbb{C}^{n_j}$  and a matrix $\tilde{v}_j$ on $\mathbb{C}^{k_j}$ such that
\begin{equation}
\label{eq:e3}
v_j  = U_j (\tilde{v}_j \otimes \id_{\cc^{m_j}}) U_j^* \quad \mu-a.s.
\end{equation}
\item\label{it:e4} There exists a full rank positive matrix $\rho_j$ on $\mathbb{C}^{k_j}$ such that
\begin{equation}
0 \oplus \dots \oplus U_j (\rho_j \otimes \id_{\cc^{m_j}}) U_j^* \oplus \dots \oplus 0 
\end{equation}
is a fixed point of $\phi$.
\end{enumerate}

It follows from \ref{it:e3} and \ref{it:e4} that the set of fixed points for $\phi$ is
\begin{equation*}
U_1\big(\rho_1\otimes M_{m_1}(\cc)\big)U_1^*\oplus\ldots \oplus U_d\big(\rho_d\otimes M_{m_d}(\cc)\big)U_d^*\oplus0_{M_D(\cc)}.
\end{equation*}
The decomposition simplifies under the purification assumption.
\begin{proposition}\label{lem:phi_FP}
Assume \pur\ holds. Then there exists a set $\{\rho_j\}_{j=1}^d$ of positive definite matrices and an integer $D$ such that the set of $\phi$ fixed points is
$$\cc\rho_1\oplus\cdots\oplus\cc\rho_d\oplus 0_{M_D(\cc)}.$$
\end{proposition}
\begin{proof}
The statement follows from the discussion preceding the proposition if we show that {\pur} implies $m_1 = \dots = m_d =1$. Assume that one of the $m_j$, e.g.\ $m_1$, is greater than $1$. Let $x$ be a norm one vector in $\mathbb{C}^{k_1}$. Then $\pi = U_j\pi_{\hat{x}} \otimes \id_{\mathbb{C}^{m_1}} U_j^*\oplus 0 \oplus \dots \oplus0$ is a projection with rank bigger than $1$, and by \Cref{eq:e3} we have
$$
 \pi v^* v \pi  = ||\tilde{v}_1 x||^2 \pi
$$
for $\mu$-almost all $v$. This contradicts {\pur}.
\end{proof}

It is clear from \Cref{eq:e2} that to each extremal fixed point $0 \oplus \dots  \oplus \rho_j \oplus \dots \oplus0$ corresponds a unique invariant measure $\nu_j$ supported on its range $F_j$.  The converse is the subject of the next proposition.
\begin{proposition}
Assume \pur\ holds. Then any $\Pi$-invariant probability measure is a convex combination of the measures $\nu_j$, $j=1,\ldots,d$.
\end{proposition}
\begin{proof}
Let $\nu$ be a $\Pi$-invariant probability measure. Let $f$ be a continuous function. From \Cref{lemma:polar},
$$\ee_\nu(f)=\lim_{n\to\infty}\ee_\nu\big(f(U_n\cdot \hat z)\big).$$
\Cref{prop:marginal} implies
$$\ee_\nu(f)=\lim_{n\to\infty}\ee^{\rho_\nu}\big(f(U_n\cdot \hat z)\big)$$
with $\rho_\nu\in\mathcal D_k$ a fixed point of $\phi$. By \Cref{lem:phi_FP}, \pur\ implies that there exist non negative numbers $t_1,\ldots,t_d$ summing up to one such that $\rho_\nu=t_1\rho_1\oplus\cdots\oplus t_d\rho_d\oplus 0_{M_D(\cc)}$. From the definition of $\pp^{\rho_\nu}$,
$$\pp^{\rho_\nu}=t_1\pp^{\rho_1}+\cdots+t_d\pp^{\rho_d}$$
where we used the abuse of notation $\rho_j\equiv 0\oplus\cdots\oplus \rho_j\oplus\cdots\oplus 0$. Using \Cref{prop:marginal}, it follows that,
$$\ee_\nu(f)=\lim_{n\to\infty} t_1\ee_{\nu_1}(f(U_n\cdot \hat z))+\cdots+t_d\ee_{\nu_d}(f((U_n\cdot \hat z))).$$
Then \Cref{lemma:polar} and the $\Pi$-invariance of each measure $\nu_j$ yield the proposition.
\end{proof}

\section{Products of special unitary matrices}\label{app:SU(k)}
\begin{proposition}\label{prop:SU(k)}
Assume $\supp\mu\subset SU(k)$. Let $G$ be the smallest closed subgroup of $SU(k)$ such that $\supp\mu\subset G$. For any $\hat x\in\sta$, let $[\hat x]_G$ be the orbit of $\hat x$ with respect to $G$ and the action $G\times \sta\ni (v,\hat x)\mapsto v\cdot \hat x$. Namely, $[\hat x]_G:=\{\hat y\in\sta \ |\ \exists v\in G \mbox{ s.t. } \hat y=v\cdot \hat x\}$. Then, for any $\hat x$, there exists a unique $\Pi$-invariant probability measure supported on $[\hat x]_G$, and this unique invariant measure is uniform in the sense that for any $v\in G$ it is invariant by the map $\hat x\mapsto v\cdot \hat x$.
\end{proposition}
\begin{corollary}
With the assumption and definitions of the last proposition, if $G=SU(k)$, $\Pi$ has a unique invariant probability measure and this probability is the uniform one on $\sta$.
\end{corollary}
\begin{proof}
The Corollary being a trivial consequence of $G=SU(k)\implies [\hat x]_G=\sta \ \forall \hat x\in \sta$, we are left with proving the Proposition.

Let $P_\mu$ be the Markov kernel on $G$ defined by the left multiplication: $P_\mu f(v)=\int_G f(uv)d\mu(u)$. Since $G$ is compact as a closed subset of $SU(k)$, following \cite[Proposition 4.8.1, Theorem 4.8.2]{Ap14}, the unique $P_\mu$-invariant probability measure $\mu_G$ on $G$ is the normalized Haar measure on $G$. Since $G$ is compact, Prokhorov’s Theorem implies that for any $u\in G$,
\begin{equation}\label{eq:SU(k)_weak}
\lim_{n\to\infty}\frac1n\sum_{k=1}^n\delta_uP_\mu^k=\mu_G\quad\mbox{weakly.}
\end{equation}
Let $\hat x\in\sta$. Since $\supp\mu\subset G$, for any $\hat y\in[\hat x]_G$, $\Pi(\hat y, [\hat x]_G)=1$. Then, $[\hat x]_G$ being compact, there exists a $\Pi$-invariant measure $\nu$ supported on $[\hat x]_G$. 

Let $f$ be a continuous function on $[\hat x]_G$. Then,
$$\nu(f)=\frac1n\sum_{k=1}^n\nu\Pi^k f=\frac1n\sum_{k=1}^n\int_{G^k\times[\hat x]_G} f(v_k \ldots v_1\cdot\hat y)d\mu\pt k(v_1,\ldots,v_k)d\nu(\hat y).$$
For each $\hat y\in[\hat x]_G$ let $u_y\in G$ be s.t. $\hat y=u_y\cdot\hat x$. The map $v\mapsto vu_y\cdot \hat x$ being continuous, setting $u=u_y$, the weak convergence \eqref{eq:SU(k)_weak} and Lebesgue's dominated convergence theorem imply,
$$\nu(f)=\int_{G}f(v\cdot \hat x) \d\mu_G(v).$$
It follows that $\nu$ is the image measure of $\mu_G$ by the application $v\mapsto v\cdot \hat x$. The left multiplication invariance of the Haar measure $\mu_G$ yields the invariance of $\nu$ by the map $\hat x\mapsto v\cdot\hat x$ for any $v\in G$.
\end{proof}
\begin{example}
Let $\mu=\frac12(\delta_{v_1}+\delta_{v_2})$ with,
$$v_1=\begin{pmatrix}
e^i&0\\0& e^{-i}
\end{pmatrix}\quad\mbox{and}\quad v_2=\begin{pmatrix}
\cos 1&i\sin 1\\ i\sin 1 &\cos 1
\end{pmatrix}.$$
Then $G=SU(2)$ and the uniform measure on $\P(\cc^2)$ is the unique $\Pi$-invariant probability measure.
\end{example}
\begin{proof}
Following \Cref{prop:SU(k)}, it is sufficient to prove that any element of $SU(2)$ is the limit of a sequence of products of $v_1$ and $v_2$.

Let $\sigma_1,\sigma_2,\sigma_3$ be the usual Pauli matrices:
$$\sigma_1:=\begin{pmatrix}
0&1\\1&0
\end{pmatrix},\quad \sigma_2:=\begin{pmatrix}
0&-i\\i&0
\end{pmatrix}\quad\mbox{and}\quad
\sigma_3=\begin{pmatrix}
1&0\\0&-1
\end{pmatrix}.$$
The Pauli matrices being generators of $SU(2)$ in its fundamental representation, for any $u\in SU(2)$, there exist three reals $\theta_1,\theta_2,\theta_3\in \rr$ s.t.,
$$u=\exp(i(\theta_1\sigma_1+\theta_2\sigma_2+\theta_3\sigma_3)).$$
Especially, $v_1=\exp(i\sigma_3)$ and $v_2=\exp(i\sigma_1)$. Since for any $i=1,2,3$, $\exp(i\theta_i\sigma_i)=\exp(i(\theta_i+2\pi)\sigma_i)$, taking limits of sequences of powers of $v_1$ or $v_2$, for any $\theta\in\rr$, both
$$e^{i\theta\sigma_1}\quad\mbox{and}\quad e^{i\theta\sigma_3}$$
are elements of $G$. It remains to show that any $u\in SU(k)$ is a product of elements equal to $\exp(i\theta\sigma_1)$ or $\exp(i\theta\sigma_3)$ with $\theta$ real.

Fix $(\theta_1,\theta_2,\theta_3)\in\rr^3$. Then using spherical coordinates in $\rr^3$, there exist $r\in\rr_+$, $\theta\in[0,\pi]$ and $\varphi\in[0,2\pi[$ such that $\theta_1=r\cos\theta$, $\theta_2=r\sin\theta\cos\varphi$ and $\theta_3=r\sin\theta\sin\varphi$. Then by direct computation,
$$e^{i(\theta_1\sigma_1+\theta_2\sigma_2+\theta_3\sigma_3)}=e^{-i\frac{\varphi}{2}\sigma_1}e^{i\frac{\theta}{2}\sigma_3}e^{ir\sigma_1}e^{-i\frac{\theta}{2}\sigma_3}e^{i\frac{\varphi}{2}\sigma_1}.$$
It follows that as a product of elements of $G$, $e^{i(\theta_1\sigma_1+\theta_2\sigma_2+\theta_3\sigma_3)}\in G$, hence $G=SU(2)$ and the example holds.
\end{proof}

\begin{example}
Let $\mu=\frac12(\delta_{v_1}+\delta_{v_2})$ with,
$$v_1=\begin{pmatrix}
i&0\\0& -i
\end{pmatrix}\quad\mbox{and}\quad v_2=\begin{pmatrix}
0&i\\ i &0
\end{pmatrix}.$$
Then $G=\{\pm \id_{\cc^2}, \pm v_1, \pm v_2, \pm v_1v_2\}$. For $z\in\cc$, let $e_z=(1,z)^\mathsf{T}$ and $e_\infty=(0,1)^\mathsf{T}$. With the conventions $\infty^{-1}=0$, $0^{-1}=\infty$ and $-\infty=\infty$, for any $z\in\cc\cup\{\infty\}$, $[\hat e_z]_G=\{\hat e_z, \hat e_{z^{-1}}, \hat e_{-z}, \hat e_{-z^{-1}}\}$ and the measure $\frac14(\delta_{\hat e_z}+\delta_{\hat e_{-z}}+\delta_{\hat e_{z^{-1}}}+\delta_{\hat e_{-z^{-1}}})$ is a $\Pi$-invariant probability measure.
\end{example}
The proof of this example is obtained by explicit computation. 
\bibliography{QTraj}
\bibliographystyle{abbrv}
\end{document}